\documentclass{tran-l}

\usepackage{amssymb}
\usepackage{graphicx}
\usepackage{verbatim}

\newtheorem{theorem}{Theorem}[section]
\newtheorem{proposition}[theorem]{Proposition}

\numberwithin{equation}{section}

\begin{document}

\title[Conelike soap films]{Conelike soap films spanning tetrahedra}

\author{Robert Huff}

\address{Department of Mathematical Sciences, Indiana University, South Bend, IN 46634}

\email{rohuff@iusb.edu}


\begin{abstract}

In this paper we provide the first examples of non-flat soap films proven to span tetrahedra.  These are members of a continuous two parameter family of soap films with tetrahedral boundaries.  Of particular interest is a two parameter subfamily where each spanning soap film has the property that two minimal surfaces meet along an edge of the boundary at an angle greater than $120^\circ$.

\end{abstract}

\maketitle

\section{Introduction}

Soap films are modeled mathematically by Almgren's $<M,0,\delta>$ minimal sets, which away from the boundary are minimal surfaces except for a singular set of 2-dimensional Hausdorff measure zero \cite{Alm76}.  Within this singular set, only two types of singularities can occur:

1.  $Y$-singularities, which are curves along which three minimal surfaces meet at $120^\circ$

2.  $T$-singularities, which are points at which four $Y$-singularities meet at\\ $\arccos\left(-\frac{1}{3}\right)\approx 109.47^\circ$.\newline
This classification of the singular set is due to Taylor \cite{Tay76}, and it verifies the experimental observations of soap films by Plateau and his students from the mid-nineteenth century.  Another property of $<M,0,\delta>$ minimal sets which is relevant here is that two minimal surfaces cannot meet along a boundary edge at an angle less than $120^\circ$ since the area along such an edge can be reduced locally.

If the boundary is tetrahedral, Lawlor and Morgan \cite{LaMo94} have shown that the flat cone over the one-skeleton of the regular tetrahedron is the least area spanning set that separates the solid tetrahedron into four regions.  However, there are currently no other soap films known to span a tetrahedral boundary.  In fact, the first mathematical existence proofs of non-flat soap films spanning any polyhedral boundary appeared only recently in \cite{rh3} and \cite{rh5}, where the boundaries under consideration are rectangular prisms with regular $n$-gon bases and each spanning soap film is homotopic to either a point or a twice-punctured sphere.  The existence of a non-flat soap film spanning a regular tetrahedron has been conjectured by Morgan.  This soap film separates the regular tetrahedron into two regions and is topologically interesting in the sense that it is homotopic to a punctured torus.  Proving its existence would provide the first example of a soap film spanning a polyhedral boundary that is homotopic to a surface of positive genus.

The boundaries considered here are tetrahedra within a two parameter family \[ \mathcal{T}=\{T_{st}\}\ , \] where $T_{st}$ is the tetrahedron with vertices \[ (\pm s,0,-\sqrt{1-s^2-t^2}/2)\ \ \mbox{and}\ \ (0,\pm t,\sqrt{1-s^2-t^2}/2)\ . \]  This family of tetrahedra can be described as follows:

$\centerdot$\ \ $T_{st}$ is symmetric with respect to reflection through the planes $x=0$ and $y=0$

$\centerdot$\ \ The origin in $\mathbb{R}^3$ is the centroid of $T_{st}$

$\centerdot$\ \ The four non-horizontal edges of $T_{st}$ have the same length

$\centerdot$\ \ The two horizontal edges of $T_{st}$ are the same if and only if $s=t$

$\centerdot$\ \ The regular tetrahedron corresponds to $s=t=1/2$.\newline Furthermore, we can identify $\mathcal{T}$ with the domain $\mathcal{Q}$ in the $st$-plane given by \begin{equation}\label{Qdef} \mathcal{Q}=\{ (s,t)\ |\ s>0, t>0 \ \ \mbox{and}\ \ s^2+t^2<1\}\ . \end{equation}

By a \textit{conelike} soap film spanning $T_{st}$ we mean a soap film homotopic to a point that contains exactly four $Y$-singularities and one $T$-singularity.  In the theorem, the four $Y$-singularities will consist of two upper curves $Y_1$, $Y_2$ in the plane $x=0$ and two lower curves $Y_3$, $Y_4$ in the plane $y=0$.  For comparison purposes, we also denote the four Euclidean segments from the $T$-singularity to the vertices of $T_{st}$ by $E_1$, $E_2$, $E_3$ and $E_4$, where $E_1$, $E_2$ are the two upper segments contained in the plane $x=0$ and $E_3$, $E_4$ are the two lower segments contained in the plane $y=0$.  The angle between $Y_1$ and $Y_2$ is always equal to the angle between $Y_3$ and $Y_4$, and the value of this common angle is $\arccos(-1/3)$.

\begin{figure}[h]\center\includegraphics[scale=0.8]{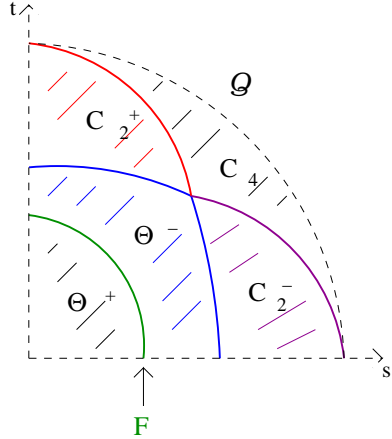}\caption{\label{Qfig}The region $\mathcal{Q}$.}\end{figure}

We will prove the existence of a conelike soap film spanning $T_{st}$ for each $(s,t)\in\mathcal{Q}$.  The qualitative properties of the spanning soap film depend on the location of $(s,t)$ relative to a partition of $\mathcal{Q}$ by sets $\Theta^+$, $\Theta^-$, $F$, $C_2^+$, $C_2^-$ and $C_4$ (See Figure \ref{Qfig}), where \begin{enumerate}
\item $\Theta^+$ consists of all points in $\mathcal{Q}$ that are strictly inside the ellipse $3s^2+3t^2+2st=2$

\item $F$ is the arc of the ellipse $3s^3+4t^2+2st=2$ that passes through $\mathcal{Q}$

\item $\Theta^-$ consists of all points in $\mathcal{Q}$ that are strictly outside the ellipse $3s^2+3t^2+2st=2$ and inside both ellipses $4s^2+3t^2=3$ and $3s^2+4t^2=3$

\item $C_2^+$ consists of all points in $\mathcal{Q}$ that are strictly outside the ellipse $3s^2+4t^2=3$ and inside the ellipse $4s^2+3t^2=3$

\item $C_2^-$ consists of all points in $\mathcal{Q}$ that are strictly outside the ellipse $4s^2+3t^2=3$ and inside the ellipse $4t^2+3s^2=3$

\item $C_4$ consists of all points in $\mathcal{Q}$ that are strictly outside both ellipses $4s^2+3t^2=3$ and $4t^2+3s^2=3$.

\end{enumerate}

With this notation, we now state the main theorem.  Figures \ref{conelike} and \ref{C4fig}, which were produced with Brakke's Surface Evolver program, show images of the soap films.

\begin{theorem}\label{main theorem}\begin{enumerate}
\item If $(s,t)\in\Theta^+$, then there exists a non-flat, conelike soap film $M_{st}$ spanning $T_{st}$ such that the angle $\arccos(-1/3)$ is greater than the angle between $E_1$ and $E_2$ and the angle between $E_3$ and $E_4$.  Furthermore, the soap films in this case can be given by explicit parameterizations.

\item If $(s,t)\in F$, then there exists a flat soap film $M_{st}$ spanning $T_{st}$.  This soap film is the cone over $T_{st}$ with vertex $p$, and $p$ is the centroid of the $T_{st}$ if and only if $s=t=1/2$.

\item If $(s,t)\in\Theta^-$, then there exists a non-flat, conelike soap film $M_{st}$ spanning $T_{st}$ such that the angle $\arccos(-1/3)$ is less than the angle between $E_1$ and $E_2$ and the angle between $E_3$ and $E_4$.  The soap films in this case can also be given by explicit parameterizations.

\item If $(s,t)\in C_2^+$, then there exists a non-flat, conelike soap film $M_{st}$ spanning $T_{st}$ such that $Y_1$ and $Y_2$ meet $T_{st}$ in the interior of its top edge, rather than at vertices.

\item If $(s,t)\in C_2^-$, then there exists a non-flat, conelike soap film $M_{st}$ spanning $T_{st}$ such that $Y_3$ and $Y_4$ meet $T_{st}$ in the interior of its bottom edge, rather than at vertices.

\item If $(s,t)\in C_4$, then there exists a non-flat, conelike soap film $M_{st}$ spanning $T_{st}$ such that all four $Y$-singularities meet $T_{st}$ in the interior of an edge.

\end{enumerate}
\end{theorem}

\begin{figure}[h]\center\includegraphics[scale=0.4]{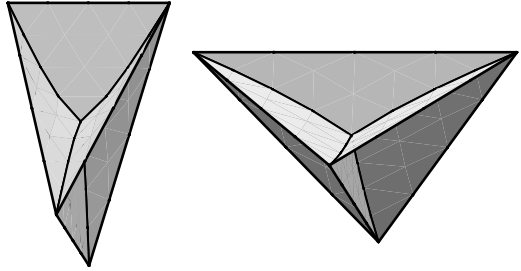}\caption{\label{conelike}On the left is a conelike soap film for $(s,t)\in\Theta^+$, and on the right is a conelike soap film for $(s,t)\in\Theta^-$.}\end{figure}

\begin{figure}[h]\center\includegraphics[scale=0.4]{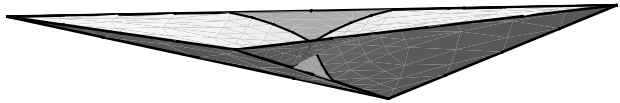}\caption{\label{C4fig}A conelike soap film for the the region $C_4$.}\end{figure}

The spanning sets in part (4) of Theorem \ref{main theorem} have the property that two minimal surfaces meet along an edge of the boundary.  Thus, if these are to be soap films we must verify that they do not meet at an angle of less than $120^\circ$.

\section{Outline of the Proof}

To prove existence in Theorem \ref{main theorem}, we derive parameterizations based on certain geometric properties of the soap films.  Each parameterization is related to conformal data on a domain in the complex plane by the following application of the famous Weierstrass Representation Theorem for minimal surfaces.

\begin{theorem}\label{wrt}  Let $\Omega\subset\mathbb{C}$ be simply connected.  If $g$ is a meromorphic function and $dh$ is a holomorphic one-form on $\Omega$ which are compatible in the sense that $g$ has a zero or pole of order $n$ at $p\in\Omega$ if and only if $dh$ has a zero of order $n$ at $p\in\Omega$, then the map $X=(X_1,X_2,X_3):\Omega\rightarrow\mathbb{R}^3$ given by
\begin{equation}\label{wrf} X(z)=\mbox{Re}\int_.^z \left(\frac{1}{2}(g^{-1}-g),\frac{i}{2}(g^{-1}+g),1\right)dh, \end{equation} is a conformal, minimal immersion.  Moreover, the function $g$ is stereographic projection of the Gauss map on the surface. \end{theorem}

Two pieces of data, the parameter domain $\Omega$ and the function $g$, are derived simultaneously.  Using the symmetries of each soap film, we determine the image of the Gauss map under stereographic projection.  This image is a domain which we take to be $\Omega$, and thus the function $g$ is the identity.

The third and final piece of data is the one form $dh$, which is a holomorphic extension of $dX_3$ and is called the \textit{complexified height differential}.  To derive this, we use a formula that relates the second fundamental form $II$ on a minimal surface to the Weierstrass data $g$ and $dh$.  In particular, for vectors $v$ and $w$ in the tangent plane to the surface at a point, we have \begin{equation}\label{2ff} \frac{dg(v)dh(w)}{g}= II(v,w)-iII(v,iw)\ . \end{equation} A nice proof of this formula as well as the statements of properties (\ref{pc}) and (\ref{ac}) can be found in \cite{hk2}.

From (\ref{2ff}) it follows that \begin{equation}\label{pc} c\ \mbox{is a principal curve}\ \Leftrightarrow\frac{dg(\dot{c})dh(\dot{c})}{g}\in\mathbb{R} \end{equation} and \begin{equation}\label{ac} c\ \mbox{is an asymptotic curve}\ \Leftrightarrow\frac{dg(\dot{c})dh(\dot{c})}{g}\in i\mathbb{R}\ . \end{equation}  Thus, we see from (\ref{pc}) and (\ref{ac}) that the function $\zeta$ given by \begin{equation}\label{zeta} \zeta(z)=\int_.^z\sqrt{\frac{dgdh}{g}} \end{equation}  maps
principal curves into vertical or horizontal lines in $\mathbb{C}$ and asymptotic curves into lines in one of the directions $e^{\pm i\pi/4}$. The map $\zeta$ is called the \textit{developing map} of the one form $\sqrt{\frac{dgdh}{g}}$.  It is a local isometry between
the minimal surface equipped with the conformal cone metric $\left|\frac{dgdh}{g}\right|$ and $\mathbb{C}$ equipped with the
Euclidean metric.

Each soap film considered in this paper has a fundamental piece whose boundary consists of principal and asymptotic curves.  This allows us
to view the function $\zeta$ as a conformal map from the parameter domain $\Omega$ onto some Euclidean polygon with determined angles and undetermined edge lengths.  Once we know $\zeta$, we can then use (\ref{zeta}) to conclude
\begin{equation}\label{dh} dh=\frac{g(d\zeta)^2}{dg}=\frac{z(d\zeta)^2}{dz}\ . \end{equation}

Now, the domain $\Omega$ is a curvilinear polygon.  Thus, the map $\zeta$ is an edge preserving conformal map between two polygons.  The existence of such a map is not always automatic and must be proven in order to finish the derivation of the parameterization.  To accomplish this, we rely on the conformal invariant extremal length, some properties of which we now describe.  For more details, see \cite{Ah73}.

Given a curvilinear polygon $\Delta$ and a conformal metric $\rho (dx^2+dy^2)$ on $\Delta$, we denote the length of a curve $\gamma\subseteq \Delta$ with respect to this metric by $\ell\rho (\gamma)$.  Similarly, we denote the $\rho$-area of a subset $U\subseteq\Delta$ by $A\rho (U)$.  Using this notation, the \textit{extremal length} between  two connected subsets, $A$ and $B$, of $\partial\Delta$ is defined by
\begin{equation} {\rm Ext}_\Delta (A,B)=\sup_\rho\frac{\inf_\gamma [\ell\rho (\gamma)]^2} {A\rho (\Delta)}\ ,\end{equation}
where the infimum is taken over all curves $\gamma: [0,1]\to\Delta$ such that $\gamma (0)\in A$, $\gamma(1)\in B$, and $\gamma(t)\subseteq
{\rm interior}(\Delta)$ for $t\in (0,1)$; the supremum is taken over all positive Borel measurable functions on $\Delta$.

Having noted these preliminaries, the properties of extremal length used below are as follows.

\begin{proposition}\label{extremal length proposition}\begin{enumerate}
\item Extremal length is invariant under biholomorphisms.

\item ${\rm Ext}_\Delta (A,B)$ is continuous with respect to $\Delta$, $A$, and $B$.

\item If $A$ and $B$ are adjacent, i.e., $\mathop{\rm dist} (A,B)=0$, then ${\rm Ext}_\Delta (A,B)=0$.

\item If $B$ is degenerate (i.e. $B$ is a point) and $\mathop{\rm dist} (A,B)>0$,
then \[{\rm Ext}_\Delta (A,B)=\infty.\]

\item If $\Delta$ is a rectangle with edges $\{B_k\}$, $k=1,2,3,4$, such that $|B_1|=|B_3|=a$ and $|B_2|=|B_4|=b$, then
\[ {\rm Ext}_\Delta(B_1,B_3)=1/{\rm Ext}_\Delta(B_2,B_4)=\frac{b}{a} \]

\item If $\Delta_1\subset\Delta_2$ are such that $A_k,B_k\subset\Delta_k$, $k=1,2$, satisfy $A_1\subset A_2$ and
$B_1\subset B_2$, then \[ {\rm Ext}_{\Delta_2}(A_2,B_2)\le {\rm Ext}_{\Delta_1}(A_1,B_1)\ , \] where
the inequality is strict if $dist(A_2,B_2)>0$ and either $A_1\neq A_2$ or $B_1\neq B_2$.

\item If $\Delta_1\subseteq\Delta_2$ are simply connected domains such that $A_k,B_k\subseteq \partial\Delta_k$, $k=1,2$ such that
every path $\gamma$ connecting $A_2$ to $B_2$ must pass through $A_1$ and
$B_1$, then \begin{equation} {\rm Ext}_{\Delta_1}(A_1,B_1)\le {\rm Ext}_{\Delta_2}(A_2,B_2)\ . \end{equation}

\end{enumerate}
\end{proposition}

\section{Proof of Theorem \ref{main theorem}}

Based on the images from Figures \ref{conelike} and \ref{C4fig}, we assume $M_{st}$ consists of two planar disks emanating from the two horizontal edges of $T_{st}$ as well as four minimal disks, which may or may not be planar, emanating from the four non-horizontal edges of $T_{st}$.  Additionally, we assume $M_{st}$, like $T_{st}$, is symmetric with respect to reflection through the planes $x=0$ and $y=0$.  Thus, to prove the existence $M_{st}$ we only need to prove the
existence of one of the minimal disks emanating from a non-horizontal edge of $T_{st}$.  We label this fundamental disk $\hat{M}_{st}$, and we assume it emanates from the edge $E_{st}$ with vertices $(0,-t,\sqrt{1-s^2-t^2}/2)$ and $(s,0,-\sqrt{1-s^2-t^2}/2)$.  In Figures \ref{conelike} and \ref{C4fig}, the disk $\hat{M}_{st}$ is the front, left, quarter of the non-flat part of each soap film.

As we will show below, the fundamental disk $\hat{M}_{st}$ will be a minimal triangle, quadrilateral or pentagon depending on the location of $(s,t)$ in $\mathcal{Q}$, but in all cases it will contain the three boundary curves $Y_1$, $Y_3$ and $E_{st}$, where

$\centerdot$\ \ $Y_1\subset\{(x,y,x_3)\ |\ x=0\}$ is a $Y$-singularity,

$\centerdot$\ \ $Y_3\subset\{(x,y,x_3)\ |\ y=0\}$ is a $Y$-singularity, and

$\centerdot$\ \ $E_{st}$ is the edge of $T_{st}$ in the direction
$v_{st}=<s,t,-\sqrt{1-s^2-t^2}>$.\\
The properties of these three curves imply that the outward
pointing normal $N$ on $\hat{M}_{st}$ makes a constant angle of
$60^\circ$ with the vector $<1,0,0>$, $<0,1,0>$ along $Y_1$, $Y_3$, respectively,
and a constant angle of $90^\circ$ with $v_{st}$ along $E_{st}$.  Thus, the
image of $N(\hat{M}_{st})$  under stereographic projection \[
\sigma=\frac{x}{1-x_3}+i\frac{y}{1-x_3} \] is such that

$\centerdot$\ \ $\sigma\circ N(Y_1)\subset\Gamma_1=\partial D(2,\sqrt{3})$

$\centerdot$\ \ $\sigma\circ N(Y_3)\subset\Gamma_2=\partial D(i2,\sqrt{3})$

$\centerdot$\ \ $\sigma\circ N(E_{st})\subset\Gamma_{st}=\partial
D\left(\frac{s+it}{\sqrt{1-s^2-t^2}},\frac{1}{\sqrt{1-s^2-t^2}}\right)\
,$\\ where $D(z,r)$ is the open disk centered at $z$ with radius
$r$.

The way the variable circle $\Gamma_{st}$ intersects the fixed circles $\Gamma_1$ and $\Gamma_2$ determines the number of boundary edges as well as the qualitative properties of $\hat{M}_{st}$.  These intersection differences are recorded in the following proposition.  Here, it is important to note the intersection points \begin{equation}\label{p1p2} p_1=1-\frac{1}{\sqrt{2}}+i\left(1-\frac{1}{\sqrt{2}}\right)\ \ \mbox{and}\ \ p_2=1+\frac{1}{\sqrt{2}}+i\left(1+\frac{1}{\sqrt{2}}\right) \end{equation} of the fixed circles $\Gamma_1$ and $\Gamma_2$ and the outer intersection points $x_0=2+\sqrt{3}$, $y_0=i(2+\sqrt{3})$, of $\Gamma_1$, $\Gamma_2$ with the real, imaginary, axis, respectively.

\begin{proposition}\label{Gammastprop}\begin{enumerate}
\item $p_1\ \ \mbox{lies strictly inside of}\ \ \Gamma_{st}$ for every $(s,t)\in\mathcal{Q}$.

\item  $(s,t)\in\Theta^+\Leftrightarrow p_2\ \ \mbox{lies strictly outside of}\ \ \Gamma_{st}$.

\item  $(s,t)\in F\Leftrightarrow  p_2\in\Gamma_{st}$

\item  $(s,t)\in\Theta^-\Leftrightarrow p_2$ lies strictly inside of $\Gamma_{st}$ and both $x_0$ and $y_0$ lie on or outside of $\Gamma_{st}$.

\item $(s,t)\in C_2^+\Leftrightarrow p_2$ lies strictly inside of $\Gamma_{st}$, $x_0$ lies on or outside of $\Gamma_{st}$ and $y_0$ lies strictly inside of $\Gamma_{st}$.

\item $(s,t)\in C_2^-\Leftrightarrow p_2$ lies strictly inside of $\Gamma_{st}$, $y_0$ lies on or outside of $\Gamma_{st}$ and $x_0$ lies strictly inside of $\Gamma_{st}$.

\item $(s,t)\in C_4\Leftrightarrow p_2$ lies strictly inside of $\Gamma_{st}$ and both $x_0$ and $y_0$ lie strictly inside of $\Gamma_{st}$.

\end{enumerate}
\end{proposition}
\begin{proof} First of all, for notational purposes we set \[ A=\sqrt{1-s^2-t^2}\ . \]  For part (1), we compute \[ \left| p_1-\left(\frac{s}{A}+i\frac{t}{A}\right)\right|^2<\frac{1}{A^2}\Leftrightarrow \left|1-\frac{1}{\sqrt{2}}-\frac{s}{A}+i\left(1-\frac{1}{\sqrt{2}}-\frac{t}{A}\right)\right|^2<\frac{1}{A^2} \] \[ \Leftrightarrow \frac{(A\sqrt{2}-A-s\sqrt{2})^2+(A\sqrt{2}-A-t\sqrt{2})^2}{2A^2}<\frac{1}{A^2}\Leftrightarrow (A\sqrt{2}-A-s\sqrt{2})^2+(A\sqrt{2}-A-t\sqrt{2})^2<1 \] \[ \Leftrightarrow (3-2\sqrt{2})A^2+s^2+t^2-(2-\sqrt{2})A(s+t)<1\Leftrightarrow 2(1-\sqrt{2})A<(2-\sqrt{2})(s+t)\ , \] and this last statement is true for every $(s,t)\in\mathcal{Q}$.

For parts (2) and (3), we compute \[ \left| p_2-\left(\frac{s}{A}+i\frac{t}{A}\right)\right|^2\ge\frac{1}{A^2}\Leftrightarrow \left|1+\frac{1}{\sqrt{2}}-\frac{s}{A}+i\left(1+\frac{1}{\sqrt{2}}-\frac{t}{A}\right)\right|^2\ge\frac{1}{A^2} \] \[ \Leftrightarrow \frac{(A\sqrt{2}+A-s\sqrt{2})^2+(A\sqrt{2}+A-t\sqrt{2})^2}{2A^2}\ge\frac{1}{A^2} \]  \[ \Leftrightarrow (1+\sqrt{2})^2A^2+s^2+t^2-\sqrt{2}(1+\sqrt{2})A(s+t)\ge 1\Leftrightarrow \sqrt{2}A\ge s+t\ , \] \[ \Leftrightarrow 2(1-s^2-t^2)\ge s^2+t^2+2st \Leftrightarrow 3s^2+3t^2+2st\le 2\ . \]  Thus, we have shown that $p_2$ lies strictly outside of $\Gamma_{st}$ if and only if $(s,t)\in\Theta^+$, and $p_2\in\Gamma_{st}$ if and only if $(s,t)\in F$.

The calculations for parts (2) and (3) also show that $p_2$ lies strictly inside $\Gamma_{st}$ if $(s,t)\in\Theta^-$.  Furthermore, we have \[ \left|x_0-\left(\frac{s}{A}+i\frac{t}{A}\right)\right|^2\ge\frac{1}{A^2}\Leftrightarrow (2+\sqrt{3})^2A^2-2(2+\sqrt{3})sA\ge A^2 \] \[ \Leftrightarrow (3+2\sqrt{3})A\ge (2+\sqrt{3})s\Leftrightarrow 4s^2+3t^2\le 3\ . \]  Similarly, we can show \[ \left|y_0-\left(\frac{s}{A}+i\frac{t}{A}\right)\right|^2\ge\frac{1}{A^2}\Leftrightarrow 3s^2+4t^2\le 3\ . \]  Thus, we have proven part (4).

Parts (5), (6) and (7) are proved similarly, and so the calculations are omitted.
\end{proof}

\subsection{Proof of parts (1) and (3)}

From parts (1) and (2) of Proposition \ref{Gammastprop} it follows that for $(s,t)\in\Theta^+$, we can take the image $\Omega_{st}=\sigma\circ
N(\hat{M}_{st})$ to be the curvilinear triangle which is the region common to the exterior of $\Gamma_{st}$ and the interiors of $\Gamma_1$ and $\Gamma_2$.  The edges of this triangle are $Y_1=\sigma\circ N(Y_1)$, $Y_3=\sigma\circ N(Y_3)$ and $E_{st}=\sigma\circ N(E_{st})$ (See Figure \ref{Thetaplus}).

\begin{figure}[h]\center\includegraphics[scale=0.7]{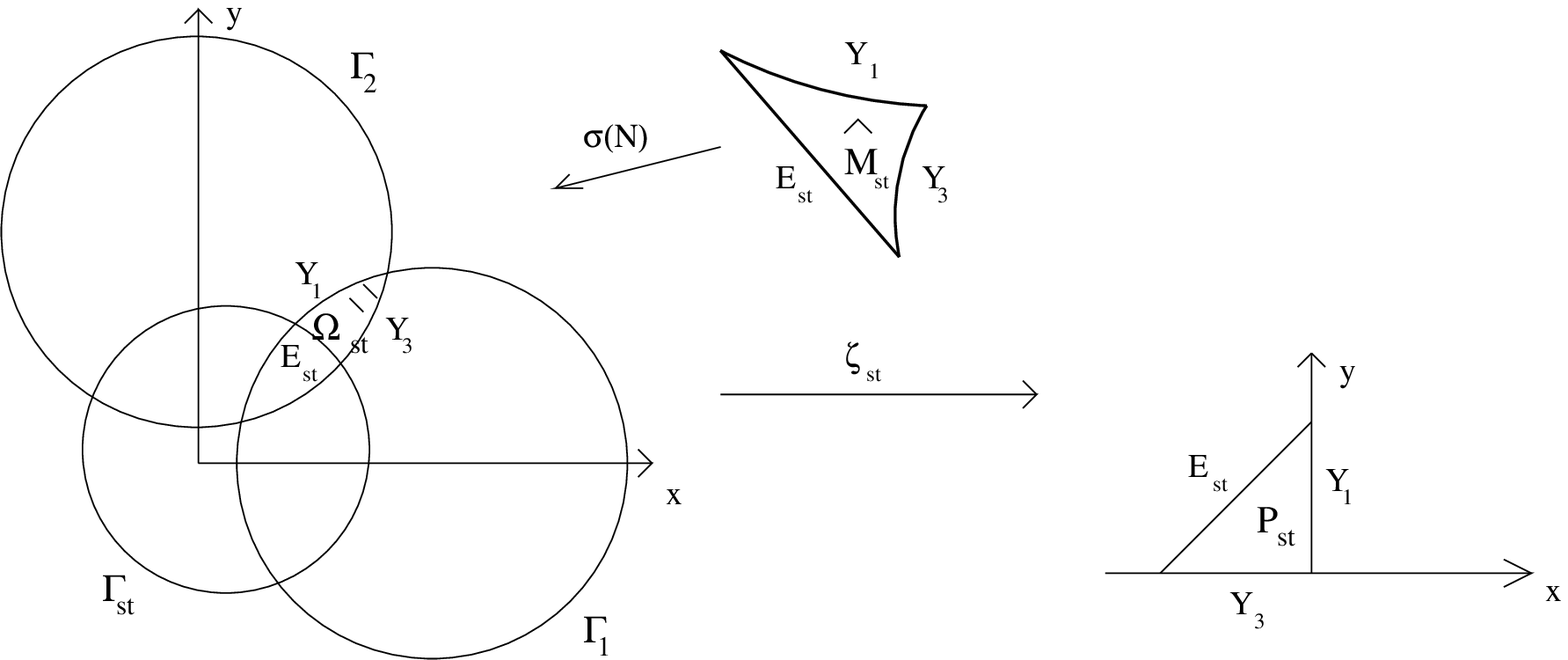}\caption{\label{Thetaplus} Center:  A fundamental piece $\hat{M}_{st}
$ in the case $(s,t)\in\Theta^+$. Left: The image $\Omega_{st}$ of the Gauss map followed by stereographic projection.  Right: The image $P_{st}$ of the map $\zeta_{st}$. }\end{figure}

From part (4) of Proposition \ref{Gammastprop}, it follows that, for $(s,t)\in\Theta^-$, we can take the Gauss image $\Omega_{st}$ to be the curvilinear triangle which is the region common to the \textit{interior} of $\Gamma_{st}$ and the \textit{exteriors} of $\Gamma_1$ and $\Gamma_2$ (See Figure \ref{Thetaminus}).

\begin{figure}[h]\center\includegraphics[scale=0.7]{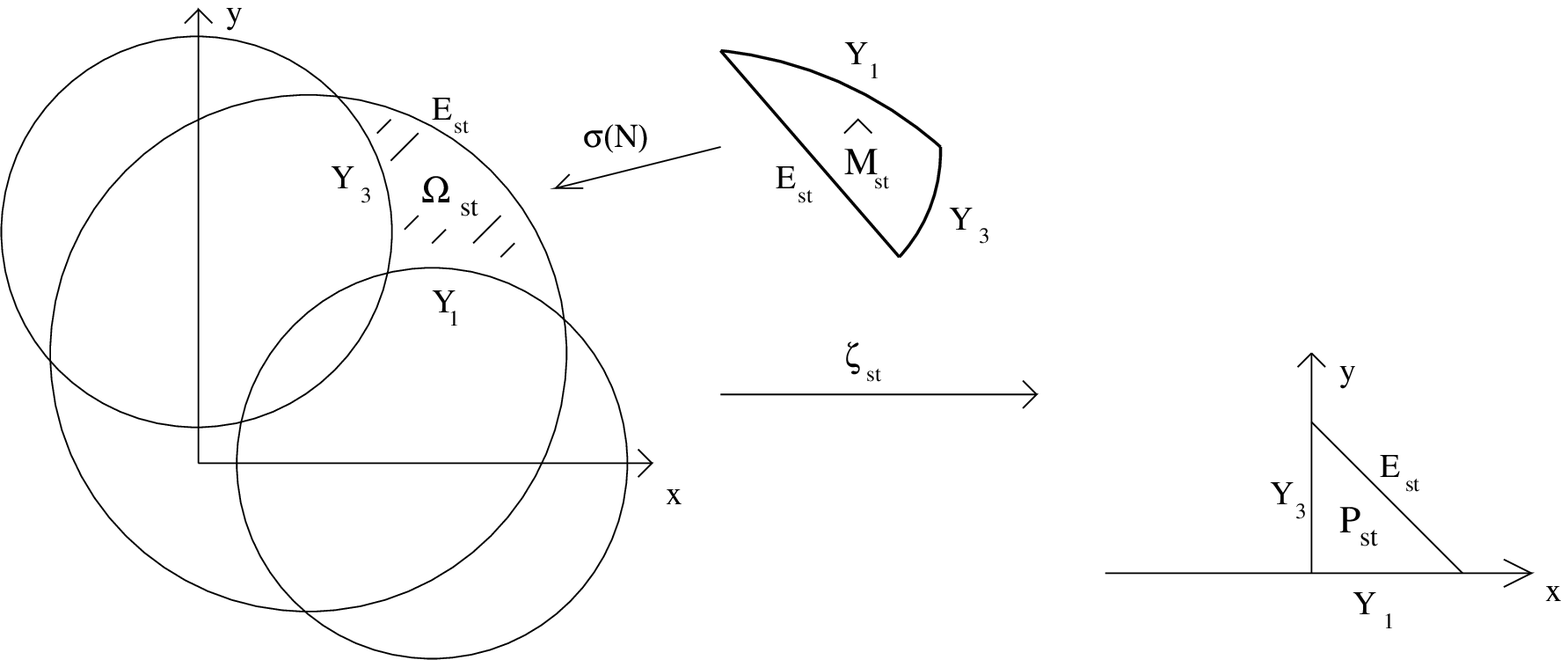}\caption{\label{Thetaminus} Center:  A fundamental piece $\hat{M}_{st}
$ in the case $(s,t)\in\Theta^-$. Left: The image $\Omega_{st}$ of the Gauss map followed by stereographic projection.  Right: The image $P_{st}$ of the map $\zeta_{st}$. }\end{figure}

To determine the image of the map $\zeta=\zeta_{st}$ defined in (\ref{zeta}), we first note the Euclidean line segment $E_{st}$ is clearly an asymptotic curve on $\hat{M}_{st}$.  Next, we have that the curves $Y_1$ and $Y_3$ are planar curves along which the Gauss map makes a constant angle with the plane of the curve.  Thus, from Joachimstahl's theorem \cite{Dc76} it follows these two curves are principal, and so we can use (\ref{pc}) and (\ref{ac}) to conclude

$\centerdot$\ \ $\zeta_{st}(Y_1)\subset\mathbb{C}$ is contained in a horizontal or vertical line,

$\centerdot$\ \ $\zeta_{st}(Y_3)\subset\mathbb{C}$ is contained in a horizontal or vertical line,

$\centerdot$\ \ $\zeta_{st}(E_{st})\subset\mathbb{C}$ is contained in a line parallel to $y=\pm x$.\\
Therefore, we conclude the image of $\zeta_{st}$ is a Euclidean triangle $P_{st}$ as shown on the right in Figures \ref{Thetaplus} and \ref{Thetaminus} with edges $Y_1=\zeta_{st}(Y_1)$, $Y_3=\zeta_{st}(Y_3)$ and $E_{st}=\zeta_{st}(E_{st})$.  In particular, if $(s,t)\in\Theta^+$, the properties $\hat{M}_{st}$ should possess imply $Y_1$ is vertical, $Y_3$ is horizontal and $E_{st}$ is parallel to $y=x$, while the opposite situation is expected if $(s,t)\in\Theta^-$.

At this point, we have derived a parameterization of $\hat{M}_{st}$ on $\Omega_{st}$ using the Weierstrass data $g(z)=z$ and $dh_{st}=(z(d\zeta_{st})^2)/dz$.  However, we have not yet proved such a parameterization exists.  To accomplish this, we must show that an edge preserving conformal map $\zeta_{st}$ exists between the fixed domain $\Omega_{st}$ and some Euclidean polygon $P_{st}$ as described above.  Now, the Riemann mapping theorem guarantees the existence of a conformal biholomorphism between any two simply connected polygons, and we have the freedom to specify the images of three vertices.  So, in this case the existence of the edge preserving conformal map $\zeta_{st}:\Omega_{st}\rightarrow P_{st}$ follows immediately.

\subsubsection{Verification of the parameterizations}

It now remains to verify that the image of our conformal, minimal immersion $X^{st}=(X_1^{st},X_2^{st},X_3^{st}):\Omega_{st}\rightarrow\mathbb{R}^3$ given by \begin{equation}\label{XCt} X^{st}(p)=\mbox{Re}\int_{p_2}^p \left(\frac{1}{2}(1-z^2),\frac{i}{2}(1+z^2),1\right)\frac{(d\zeta_{st})^2}{dz} \end{equation} is indeed the expected fundamental piece $\hat{M}_{st}$.  In what follows, we will use this form of the Weierstrass representation formula rather than that given in (\ref{wrf}).  Also, notice that we are using the intersection point $p_2$ of $Y_1\subset\Gamma_1$ and $Y_3\subset\Gamma_2$ as the base point of integration.

First, we need to show the image of $X^{st}$ is compact.  This can be accomplished if we can show the three one forms in (\ref{XCt}) are integrable on $\Omega_{st}$, and integrability will fail only if the one form $(d\zeta_{st})^2/dz$ has a non-integrable singularity at one of the three vertices.

At the vertex $p_2=Y_1\cap Y_3$, the map $\zeta_{st}$ takes an angle of $\psi_0=\arccos(1/3)$ on $\Omega_{st}$ to an angle of $\pi/2$ on $P_{st}$.  Thus, near $p_2$ we have $(d\zeta_{st})^2/dz=\xi(z)(z-p_2)^{(\pi-2\psi_0)/\psi_0}dz$, where $\xi(z)$ is holomorphic and non-zero on a neighborhood of $p_2$. Since the exponent $(\pi-2\psi_0)/\psi_0>-1$, it follows that $(d\zeta_{st})^2/dz$ is integrable at $p_2$.  Furthermore, we record here that $X^{st}$ takes the angle $\psi_0$ on $\Omega_{st}$ to an angle \begin{equation}\label{theta0} \theta_0=\arccos(-1/3) \end{equation} on the image $X^{st}(\Omega_{st})\subset\mathbb{R}^3$.  This is exactly what we expected since $X^{st}(p_2)$ should be a $T$-singularity.

Each of the remaining two angles $\psi_1$ and $\psi_2$ is mapped to an angle of $\pi/4$ on $P_{st}$.  Thus, the one form $(d\zeta_{st})^2/dz$ will be integrable at $v_1=Y_1\cap E_{st}$ and $v_2=Y_3\cap E_{st}$ if and only if $\psi_1,\psi_2\neq \pi/2$ for $(s,t)\in\Theta^+\cup\Theta^-$.  In fact, we can prove even more, which we do now in the following proposition.

\begin{proposition}\label{GammastGamma12} $\Gamma_{st}$ does not intersect $\Gamma_1$ or $\Gamma_2$ orthogonally for any $(s,t)\in\mathcal{Q}$.  \end{proposition}
\begin{proof} The circles $\Gamma_{st}$ and $\Gamma_1$ will be orthogonal if and only if the scalar product $<x_{st}-2,y_{st}>\cdot<x_{st}-s/A,y_{st}-t/A>$ is zero for some intersection point $(x_{st},y_{st})\in\Gamma_{st}\cap\Gamma_1$, where as before we set $A=\sqrt{1-s^2-t^2}$.  Computing, we have that this scalar product is equal to \begin{equation}\label{sp1} x_{st}^2-x_{st}\frac{s}{A}-2x_{st}+2\frac{s}{A}+y_{st}^2-y_{st}\frac{t}{A}\ . \end{equation}  Since $v_1\in\Gamma_1$, we have $x_{st}^2+y_{st}^2=-1+4x_{st}$.  Making this substitution into the expression (\ref{sp1}), we have \begin{equation}\label{sp2} <x_{st}-2,y_{st}>\cdot<x_{st}-s/A,y_{st}-t/A>\ =\ 4x_{st}-1-\frac{sx_{st}+ty_{st}}{A}-2x_{st}+2\frac{s}{A}\ . \end{equation}  At this point, we need the equation \[ sx_{st}+ty_{st}=A(2x_{st}-1)\ , \] which follows from the fact that $v_1$ is an intersection point of $\Gamma_{st}$ and $\Gamma_1$.  Incorporating this relationship into the right hand side of (\ref{sp2}), we have \begin{equation}\label{sp3}  <x_{st}-2,y_{st}>\cdot<x_{st}-s/A,y_{st}-t/A>\ =2\frac{s}{A}\ , \end{equation} and this is non-zero for all $(s,t)\in\mathcal{Q}$.  Similarly, we can also show $\Gamma_{st}$ does not intersect $\Gamma_2$ orthogonally for any $(s,t)\in\mathcal{Q}$.
\end{proof}

Now, the angles $\psi_1$ and $\psi_2$ are continuous on all of $\mathcal{Q}$.  Furthermore, we can compute at $(s,t)=(1/2,1/2)$ that $\psi_1=\psi_2=\arccos(1/\sqrt{3})$, and this angle is less than $\pi/2$.  Thus, from Proposition \ref{GammastGamma12} it follows that $\psi_1$ and $\psi_2$ are less than $\pi/2$ for all $(s,t)\in\mathcal{Q}$.  Therefore, the one form $(d\zeta_{st})^2/dz$ is integrable at $v_1$ and $v_2$, and we have shown \begin{equation}\label{compactimage} X^{st}(\Omega_{st})\ \mbox{is compact}\ . \end{equation}

Next, we analyze $X^{st}$ on $\partial\Omega_{st}$ to ensure the boundary of the image in $\mathbb{R}^3$ has the geometric properties we expect. Beginning with $Y_1$ and $(s,t)\in\Theta^+$, we parameterize in the counterclockwise direction from $p_2$ to $v_1$ by $z_1(w)=2+\sqrt{3}e^{iw}$.  The value of the parameter $w$ at $p_2$ is $\arccos((1-\sqrt{2})/\sqrt{6})$, while the value at $v_1$ is always less than $(5\pi)/6$.  To prove the former statement, we just compute the angle $p_2-2=1/\sqrt{2}-1+i(1/\sqrt{2}+1)$ makes with the positive $x$-axis.  To prove the latter statement, we first show that $\Gamma_{st}$ always contains the points $x+iy$ on the unit circle where $x>0$ and $y>0$.  Computing, we have \[ \left|x-\frac{s}{A}+i\left(y-\frac{t}{A}\right)\right|^2<\frac{1}{A^2}\Leftrightarrow xs+yt>0\ , \] and this is true if $s,t,x,y>0$.  Thus, the value of $w$ at $v_1$ is less than the value of $w$ at the intersection point $z=1/2+i\sqrt{3}/2$ of $\Gamma_1$ with the unit circle, and this value is $(5\pi)/6$.  Therefore, we have that \[ \arccos((1-\sqrt{2})/\sqrt{6})\le w< (5\pi)/6\ . \]

Furthermore, we have $dz(\dot{z}_1)=i\sqrt{3}e^{iw}$ and $d\zeta_{st}(\dot{z}_1)^2< 0$ on the interior of $Y_1$.  Computing, we have \[ dX_1^{st}(\dot{z}_1)=\mbox{Re}\frac{1}{2}\left((1-z_1^2)\frac{d\zeta_{st}(\dot{z}_1)^2}{dz(\dot{z}_1)}\right)= \] \begin{equation}\label{Xt1z1} = -\ \frac{(d\zeta_{st}(\dot{z}_1)^2}{2\sqrt{3}}\mbox{Re}(ie^{-iw}(-3-4\sqrt{3}e^{iw}-3e^{i2w}))= 0\ . \end{equation}  Continuing, we have \[ dX_2^{st}(\dot{z}_1)=\mbox{Re}\frac{1}{2}\left(i(1+z_1^2)\frac{d\zeta_{st}(\dot{z}_1)^2}{dz(\dot{z}_1)}\right)= \] \begin{equation}\label{Xt2z1} = \frac{d\zeta_{st}(\dot{z}_1)^2}{2\sqrt{3}}\mbox{Re}(e^{-iw}(5+4\sqrt{3}e^{iw}+3e^{i2w}))= \frac{2d\zeta_{st}(\dot{z}_1)^2}{\sqrt{3}}(2\cos w+\sqrt{3}) < 0  \end{equation} on the interior of $Y_1$.  The inequality at the end of (\ref{Xt2z1}) follows from the fact that the expression $2\cos w+\sqrt{3}$ is always positive since $\arccos((1-\sqrt{2})/\sqrt{6})\le w <(5\pi)/6$.  For the $x_3$ component, we have \[ dX_3^{st}(\dot{z}_1)=\mbox{Re}\left(\frac{z_1d\zeta_{st}(\dot{z}_1)^2}{dz(\dot{z}_1)}\right)= \] \begin{equation}\label{Xt3z1} = -\ \frac{d\zeta_{st}(\dot{z}_1)^2}{\sqrt{3}}\mbox{Re}(ie^{-iw}(2+\sqrt{3}e^{iw}))= -\ \frac{2d\zeta_{st}(\dot{z}_1)^2}{\sqrt{3}}(\sin w) > 0 \end{equation} on the interior of $Y_1$.

Equations (\ref{Xt1z1})-(\ref{Xt3z1}) imply $X^{st}(Y_1)$ is the graph of some decreasing function $g_1$ in the $yx_3$-plane, where $X^{st}_3=g_1(X^{st}_2)$.  Furthermore, we can compute \[ g_1'(X^{st}_2)=\frac{(X^{st}_3)'}{(X^{st}_2)'} = -\ \frac{\sin w}{2\cos w + \sqrt{3}}\ , \] and so \begin{equation}\label{g1''} g_1^{\prime\prime}(X^{st}_2)=\frac{((X^{st}_3)'/(X^{st}_2)')'}{(X^{st}_2)'} = -\ \frac{2+\sqrt{3}\cos w}{(2\cos w+ \sqrt{3})^2(X^{st}_2)'} > 0 \end{equation} on the interior of $Y_1$.  Therefore, we have shown \[ X^{st}(Y_1), (s,t)\in\Theta^+,\ \mbox{is the graph of a decreasing,} \] \begin{equation}\label{Xtf1} \mbox{concave upward function in the}\ yx_3-\mbox{plane}\ . \end{equation}

For $(s,t)\in\Theta^-$, we have that $z_1$ parameterizes $Y_1$ from $v_1$ to $p_2$ and $d\zeta_{st}(\dot{z_1})^2> 0$ on the interior of $Y_1$.  The parameter $w$ in this case is such that \[ 0\le w\le\arccos((1-\sqrt{2})/\sqrt{6})\ . \]  In fact, the property $w\ge 0$ is the defining property of $\Theta^-$, as we will see below.  The formulas for our calculations are the same as in (\ref{Xt1z1})-(\ref{g1''}), but the inequalities are reversed and the conclusion is  \[ X^{st}(Y_1), (s,t)\in\Theta^-,\ \mbox{is the graph of a decreasing,} \] \begin{equation}\label{Xtf1a} \mbox{concave downward function in the}\ yx_3-\mbox{plane}\ . \end{equation}  In particular, to obtain the reverse inequality at the end of (\ref{Xt3z1}), it is crucial that $4s^2+3t^2\le 3$.  For, when $4s^2+3t^2>3$ we have that $\Gamma_{st}$ intersects $\Gamma_1$ \textit{below} the $x$-axis.  Thus, the parameter $w$ takes on negative values near zero, and this implies the expression on the right in (\ref{Xt3z1}) is positive near the vertex $X^{st}(v_1)$ of the tetrahedron $T_{st}$.  This in turn implies the curve $X^{st}(Y_1)$ is increasing near $X^{st}(v_1)$, which means we can no longer extend the fundamental piece to a soap film spanning $T_{st}$.  If $4s^2+3t^2=3$, then $w$ takes on the value zero at the vertex $v_1$, and (\ref{Xt2z1}) and (\ref{Xt3z1}), with the inequalities reversed, imply $Y_1$ meets the top edge of $T_{st}$ tangentially.

In both cases, we have  \[ dX^{st}(\dot{z}_1)=\frac{d\zeta_{st}(\dot{z}_1)^2}{\sqrt{3}}\left<0,\frac{1+\sqrt{2}}{\sqrt{3}},-\ \frac{1+\sqrt{2}}{\sqrt{6}}\right> \] at $w=\arccos((1-\sqrt{2})/\sqrt{6})$.  From this we compute that the angle between the $x_3$-axis and the tangent line to $Y_1$ at $X^{st}(p_2)$ is $\theta_0^x=\arccos(1/\sqrt{3})$.  Thus, we have \begin{equation}\label{theta0x} 2\theta_0^x=\arccos(-1/3)\ , \end{equation} where $2\theta_0^x$ is the angle between $X^{st}(Y_1)$ and its image under reflection through the $xx_3$-plane.

We next parameterize the curve $Y_3$ for $(s,t)\in\Theta^+$ in the counterclockwise direction from $v_2$ to $p_2$ by $z_2(w)=i2+\sqrt{3}e^{iw}$, and in this case the parameter $w$ satisfies the inequality \[ -\pi/3<w\le\pi/2-\arccos((1-\sqrt{2})/\sqrt{6})\ . \]  Here, we have that $d\zeta_{st}(\dot{z}_2)^2> 0$ on the interior of $Y_3$.  Calculating as above, we have \[ dX_1^{st}(\dot{z}_2)=\mbox{Re}\frac{1}{2}\left((1-z_2^2)\frac{d\zeta_{st}(\dot{z}_2)^2}{dz(\dot{z}_2)}\right)= \] \begin{equation}\label{Xt1z2} = -\ \frac{d\zeta_{st}(\dot{z}_2)^2}{2\sqrt{3}}\mbox{Re}(ie^{-iw}(5-i4\sqrt{3}e^{iw}-3e^{i2w}))= -\ \frac{2d\zeta_{st}(\dot{z}_2)^2}{\sqrt{3}}(2\sin w+\sqrt{3})< 0  \end{equation} on the interior of $Y_3$.  Continuing, we have \[ dX_2^{st}(\dot{z}_2)=\mbox{Re}\frac{1}{2}\left(i(1+z_2^2)\frac{d\zeta_{st}(\dot{z}_2)^2}{dz(\dot{z}_2)}\right)= \] \begin{equation}\label{Xt2z2} = \frac{d\zeta_{st}(\dot{z}_2)^2}{2\sqrt{3}}\mbox{Re}(e^{-iw}(-3+i4\sqrt{3}e^{iw}+3e^{i2w}))= 0\ . \end{equation}  For the $x_3$ component, we have \[ dX_3^{st}(\dot{z}_2)=\mbox{Re}\left(\frac{z_2d\zeta_{st}(\dot{z}_2)^2}{dz(\dot{z}_2)}\right)= \] \begin{equation}\label{Xt3z2} = -\ \frac{d\zeta_{st}(\dot{z}_2)^2}{\sqrt{3}}\mbox{Re}(ie^{-iw}(i2+\sqrt{3}e^{iw}))= \frac{2d\zeta_{st}(\dot{z}_1)^2}{\sqrt{3}}(\cos w) > 0 \end{equation} on the interior of $Y_3$.

Equations (\ref{Xt1z2})-(\ref{Xt3z2}) imply $X^{st}(Y_3)$ is the graph of some decreasing function $g_2$ in the $xx_3$-plane, where $X^{st}_3=g_2(X^{st}_1)$.  Furthermore, we can compute \[ g_2'(X^{st}_1)=\frac{(X^{st}_3)'}{(X^{st}_1)'} = -\ \frac{\cos w}{2\sin w + \sqrt{3}}\ , \] and so \begin{equation}\label{g2''} g_2^{\prime\prime}(X^{st}_1)=\frac{((X^{st}_3)'/(X^{st}_1)')'}{(X^{st}_1)'} = \frac{2+\sqrt{3}\sin w}{(2\sin w+ \sqrt{3})^2(X^{st}_1)'} < 0 \end{equation} on the interior of $Y_3$.  Therefore, we have shown \[ X^{st}(Y_3), 1/2<t\le\sqrt{3/7},\ \mbox{is the graph of a decreasing,} \] \begin{equation}\label{Xtf2} \mbox{concave downward function in the}\ xx_3-\mbox{plane}\ . \end{equation}

For $(s,t)\in\Theta^-$, we have that $z_2$ parameterizes $Y_3$ from $p_2$ to $v_2$ and $d\zeta_{st}(\dot{z_2})^2< 0$ on the interior of $Y_3$.  The parameter $w$ in this case is such that \[ \pi/2-\arccos((1-\sqrt{2})/\sqrt{6})\le w \le \pi/2\ . \]  The formulas for our calculations are the same as in (\ref{Xt1z2})-(\ref{g2''}), but the inequalities are reversed and the conclusion is  \[ X^{st}(Y_3), (s,t)\in\Theta^-,\ \mbox{is the graph of a decreasing,} \] \begin{equation}\label{Xtf2a} \mbox{concave upward function in the}\ xx_3-\mbox{plane}\ . \end{equation}  Similar to $Y_1$, we cannot extend our fundamental piece to a soap film spanning $T_{st}$ if $3s^2+4t^2>3$, and if $3s^2+4t^2=3$ we have that $Y_3$ meets the bottom edge of $T_{st}$ tangentially.

As with $Y_1$, in both cases we have  \[ dX^{st}(\dot{z}_2)=\frac{d\zeta_{st}(\dot{z}_2)^2}{\sqrt{3}}\left<-\ \frac{1+\sqrt{2}}{\sqrt{3}},0,\frac{1+\sqrt{2}}{\sqrt{6}}\right> \] at $w=\pi/2-\arccos((1-\sqrt{2})/\sqrt{6})$.  From this we compute that the angle between the $x_3$-axis and the tangent line to $Y_3$ at $X^{st}(p_2)$ is $\theta_0^y=\arccos(1/\sqrt{3})$.  Thus, we have \begin{equation}\label{theta0y} 2\theta_0^y=\arccos(-1/3)\ , \end{equation} where $2\theta_0^y$ is the angle between $X^{st}(Y_1)$ and its image under reflection through the $yx_3$-plane.  Therefore, combining equations (\ref{theta0}), (\ref{theta0x}) and (\ref{theta0y}), we have that the point \begin{equation}\label{Tsingularity} X^{st}(p_2)\ \ \mbox{is a}\ T-\mbox{singularity}\ . \end{equation}

The third and final curve to verify is $E_{st}$, which for $(s,t)\in\Theta^+$ we parameterize in the counterclockwise direction from $v_2$ to $v_1$ by $z_3(w)=(s+it)/\sqrt{1-s^2-t^2}+e^{iw}/\sqrt{1-s^2-t^2}$.  Here, we have $\dot{z}_3(w)=ie^{iw}/\sqrt{1-s^2-t^2}$, and $d\zeta_{st}(\dot{z}_3)^2=i|d\zeta_{st}(\dot{z}_3)|^2$.  Computing, we have \[ dX_1^{st}(\dot{z}_3)= -\ \frac{|d\zeta_{st}(\dot{z}_3)|^2}{2\sqrt{1-s^2-t^2}}\mbox{Re}(e^{-iw}(1-2s^2-e^{i2w}-i2st-2se^{iw}-i2te^{iw}))= \] \begin{equation}\label{Xt1z3} =-\ \frac{s|d\zeta_{st}(\dot{z}_3)|^2}{\sqrt{1-s^2-t^2}}(1+s\cos w+t\sin w)<0 \end{equation} on the interior of $E_{st}$.  The inequality follows because the minimum value of $f(w)=1+s\cos w+t\sin w$ is $1-\sqrt{s^2+t^2}$.  Thus, we have that $f(w)$ is always positive if $(s,t)\in\mathcal{Q}$.

Continuing, we have \[ dX_2^{st}(\dot{z}_3)= \frac{|d\zeta_{st}(\dot{z}_3)|^2}{2\sqrt{1-s^2-t^2}}\mbox{Re}(ie^{-iw}(1-2t^2+e^{i2w}+i2st+2se^{iw}+i2te^{iw}))= \] \begin{equation}\label{Xt2z3} = -\ \frac{t|d\zeta_{st}(\dot{z}_3)|^2}{\sqrt{1-s^2-t^2}}(1+s\cos w+t\sin w)<0 \end{equation} on the interior of $E_{st}$.  For the $x_3$ component, we have \[ dX_3^{st}(\dot{z}_3)=|d\zeta_{st}(\dot{z}_3)|^2\mbox{Re}(e^{-iw}(s+it+e^{iw}))= \] \begin{equation}\label{Xt3z3} = |d\zeta_{st}(\dot{z}_3)|^2(1+s\cos w+t\sin w)>0 \end{equation} on the interior of $E_{st}$.

From equations (\ref{Xt1z3})-(\ref{Xt3z3}) we have \[ \left<dX^{st}_1(\dot{z}_3),dX^{st}_2(\dot{z}_3),dX^{st}_3(\dot{z}_3)\right>= \] \[ -\ \frac{|d\zeta_{st}(\dot{z}_3)|^2}{\sqrt{1-s^2-t^2}}(1+s\cos w+t\sin w)\left<s,t,-\sqrt{1-s^2-t^2}\right>\ . \]  Thus, for $(s,t)\in\Theta^+$, the immersion $X^{st}$ maps $E_{st}\subset\Omega_{st}$ monotonically onto a Euclidean line segment in the direction $v_{st}=<s,t,-\sqrt{1-s^2-t^2}>$ of the edge $E_{st}$ of the tetrahedron $T_{st}$.  If we want $X^{st}(E_{st})$ to have length one, we simply scale $\mathbb{R}^3$ by the appropriate constant $\lambda$, which is equivalent to scaling $P_{st}$ and $\zeta_{st}$ by $\sqrt{\lambda}$.

If $(s,t)\in\Theta^-$, then $z_3$ parameterizes $E_{st}$ from $v_1$ to $v_2$ and $d\zeta_{st}(\dot{z}_3)^2=-i|d\zeta_{st}(\dot{z}_3)|^2$.  In this case, the calculations are similar and the conclusions are exactly the same as for $(s,t)\in\Theta^+$.  So, we omit their repetition.

We must now show our fundamental piece $\hat{M}_{st}=X^{st}(\Omega_{st})$ does not have any self intersections and that no intersections will be introduced upon extension to a soap film $M_{st}$ spanning $T_{st}$.  From the above calculations, we have that the boundary of $X^{st}(\Omega_{st})$ has a one to one projection onto the boundary of a convex polygon in the $xy$-plane.  Thus, by a theorem of Rad\'{o} \cite{ra1} it follows that the minimal surface $X^{st}(\Omega_{st})$ is a graph over this convex polygon.  In particular, the surface has no self intersections and no intersections are introduced upon extension.

Finally, from statements (\ref{Xtf1}) and (\ref{Xtf2}) it follows that $\arccos(-1/3)$ is greater than the angle between $E_1$ and $E_2$ and the angle between $E_3$ and $E_4$ for $(s,t)\in\Theta^+$.  Similarly, statements (\ref{Xtf1a}) and (\ref{Xtf2a}) imply $\arccos(-1/3)$ is less than the angle between $E_1$ and $E_2$ and the angle between $E_3$ and $E_4$ for $(s,t)\in\Theta^-$.  Here, the Euclidean segments $E_1$, $E_2$, $E_3$ and $E_4$ are the four segments from the $T$-singularity $X^{st}(p_2)$ to the vertices of $T_{st}$.

\subsubsection{Explicit parameterizations}

We first change coordinates via the conformal map $\Phi_{st}:\mathbb{H}\rightarrow\Omega_{st}$ normalized so that \[ \Phi_t(0)=p_2,\ \Phi_t(1)=v_1\ \ \mbox{and}\ \ \Phi_t(\infty)=v_2\ . \]  This gives a parameterization $Z^{st}=(Z_1^{st},Z_2^{st},Z_3^{st})$ for $\hat{M}_{st}$ on the upper half plane $\mathbb{H}$ with Weierstrass data \[ g_{st}=\Phi_{st}\ \mbox{and}\ dh_{st}=\frac{\Phi_{st}(d\Psi_{st})^2}{d\Phi_{st}}\ , \] where $\Psi_{st}=\zeta_{st}\circ\Phi_{st}$ is the conformal map from $\mathbb{H}$ onto the triangle $P_{st}$ normalized so that \[ \Psi_{st}(0)=0,\ \Psi_{st}(1)=v_1\ \mbox{and}\ \Psi_{st}(\infty)=v_2\ . \]  Since $P_{st}$ is Euclidean, the map $\Psi_{st}$ is a Schwarz-Christoffel map given by \[ \Psi_{st}(z)=A_{st}\int_0^zw^{-1/2}(w-1)^{-3/4}dw\ , \] where the constant \[ A_{st}=|A_{st}|e^{-i3\pi/4},\ \ (s,t)\in\Theta^+ \] \[ A_{st}=|A_{st}|e^{i3\pi/4},\ \ (s,t)\in\Theta^- \] is determined by the equation $|Z^{st}([1,\infty))|=|E_{st}|=1$.  Furthermore, the map $\mathcal{M}_{st}\circ\Phi_{st}$ can be made explicit in terms of hypergeometric functions since the domain $\Omega_{st}$ for $(s,t)\in\Theta^+\cup\Theta^-$ is a curvilinear triangle.  The map $\mathcal{M}_{st}$ is a M\"{o}bius transformation that normalizes $\Omega_{st}$ so that $p_2$ is the origin in $\mathbb{C}$ and the curves $Y_1$ and $Y_3$ are Euclidean segments with $Y_1$ lying along the positive $x$-axis.  With this normalization, the reader is referred to \cite{ca54} or \cite{hmc1} for the explicit formulas for $\mathcal{M}_{st}\circ\Phi_{st}$.

\subsection{Proof of parts (4), (5) and (6)}

If $(s,t)\in C_4$, then from part (7) of Proposition \ref{Gammastprop} it follows that $\Gamma_{st}$ does not intersect $\Gamma_1$ above the $x$-axis or $\Gamma_2$ to the right of the $y$-axis.  We could still define $\Omega_{st}$ as we did for $(s,t)\in\Theta^-$, and in doing so we would obtain a fundamental piece $\hat{M}_{st}$.  However, as we saw in the above proof of part (3), this fundamental piece would not extend to a soap film spanning $T_{st}$.  This will happen anytime the vertex $v_1\in\Gamma_{st}\cap\Gamma_1$ lies strictly below the $x$-axis or the vertex $v_2\in\Gamma_{st}\cap\Gamma_2$ lies strictly to the left of the $y$-axis.

We can get around this obstacle by introducing two edges into $\hat{M}_{st}$.  This is done by allowing the $Y$-singularities $Y_1$ and $Y_3$ to meet the top and bottom edges, respectively, of $T_{st}$ \textit{at an interior point} rather than a vertex.  Thus, we now have a subsegment of these tetrahedral edges contained in the boundary of $\hat{M}_{st}$.  The effect this has on the image of the Gauss map $\Omega_{st}$ is the introduction of two Euclidean line segments into its boundary - a horizontal segment $E^x_{st}$ along the $x$-axis that connects $\Gamma_1$ with $\Gamma_{st}$, and a vertical segment $E^y_{st}$ along the $y$-axis that connects $\Gamma_2$ and $\Gamma_{st}$.  So, in this case we define $\Omega_{st}$ to be the curvilinear pentagon that is the region common to the interior of $\Gamma_{st}$, the exteriors of $\Gamma_1$ and $\Gamma_2$, and the first quadrant in the complex plane (See Figure \ref{C4diagram}).

In this case, the spanning sets will have the property that two minimal surfaces meet along portions of the top or bottom edge of $T_{st}$.  We must show they meet at an angle greater than or equal to $120^\circ$.  For this, notice that if $(s,t)$ is on the boundary of $C_4$ and $C_2^+\cup C_2^-$ in $\mathcal{Q}$, then $v_1=2+\sqrt{3}$ or $v_2=i(2+\sqrt{3})$.  These correspond to the values where the outward pointing normal $N$ on the surface makes an angle of $30^\circ$ with the vector $<0,0,1>$.  Furthermore, this angle on the surface at $v_1$ or $v_2$ is less than $30^\circ$ for $v_1>2+\sqrt{3}$ or $-iv_2>2+\sqrt{3}$, since this corresponds to an upward rotation of the outward pointing normal.  This implies that $120^\circ$ is a lower bound for the angle at which two minimal surface meet along a top or bottom edge of $T_{st}$ for $(s,t)\in C_4$.  Thus, the spanning sets $M_{st}$ in this region will indeed be soap films.

For the image of $\zeta_{st}$, notice that the two boundary edges introduced in $\hat{M}_{st}$ are asymptotic curves.  Thus, we modify the triangle $P_{st}$ for $(s,t)\in\Theta^-$ by introducing two edges $E^x_{st}$ and $E^y_{st}$ parallel to the line $y=x$.  One of these edges connects $Y_1$ and $E_{st}$, while the other connects $Y_3$ and $E_{st}$.  Therefore, the map $\zeta_{st}$ is an edge-preserving conformal map between the curvilinear polygon $\Omega_{st}$ and some Euclidean pentagon $P_{st}$ whose edges are in the directions mentioned above. (See Figure \ref{C4diagram}).  The existence of such a map does not follow solely from the Riemann mapping theorem since the polygons under consideration in this case are five-sided instead of three-sided.

\begin{figure}[h]\center\includegraphics[scale=0.7]{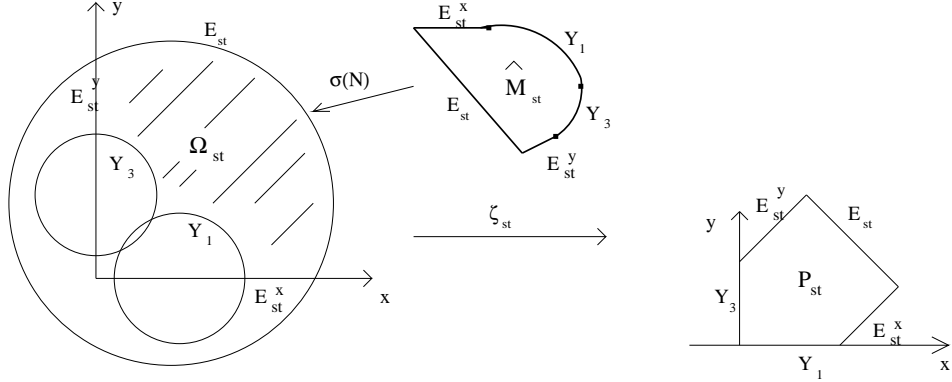}\caption{\label{C4diagram} Center:  A fundamental piece $\hat{M}_{st}
$ in the case $(s,t)\in C_4$. Left: The image $\Omega_{st}$ of the Gauss map followed by stereographic projection.  Right: The image $P_{st}$ of the map $\zeta_{st}$. }\end{figure}

To prove the existence of $\zeta_{st}$, consider the space $\mathcal{P}$ of Euclidean pentagons with edges oriented and labeled like $P_{st}$ in Figure \ref{C4diagram} and normalized so that $Y_1\cap Y_3$ is the origin in $\mathbb{C}$ and $|Y_1|=1$.  With this normalization, each pentagon $P=P_{\ell m}\in\mathcal{P}$ is uniquely determined by the lengths $\ell=|E^x_{st}|$ and $m=|E_{st}|$.  Thus, we can identify the space $\mathcal{P}$ with the domain \[ \mathcal{D}=\{ (\ell,m)\ |\ \ell>0\ \ \mbox{and}\ \ 1/\sqrt{2}<m<\ell+\sqrt{2}\} \] in the $\ell m$-plane.  With this notation, we state the following proposition.

\begin{proposition}\label{mainprop} For each $(s,t)\in C_4$, there exists an edge preserving
conformal map $\zeta_{s,t}$ from $\Omega_{s,t}$ onto some Euclidean polygon $P_{\ell m}\in\mathcal{P}$.
\end{proposition}
$\mathit{proof}$:  Fix $(s,t)\in C_4$ and $0<\ell<\infty$.  Then $|Y_3|\rightarrow 0$ and $|E^y_{st}|\rightarrow\ell+1/\sqrt{2}>0$ on $P_{\ell m}$ as $m\rightarrow 1/\sqrt{2}$, and so it follows from parts (2) and (3) of Proposition \ref{extremal length proposition} that
\[ Ext_{P_{\ell m}}(Y_1,E^y_{st})\rightarrow 0\ \ \mbox{as}\ \
m\rightarrow 1/\sqrt{2}\ . \] As $m\rightarrow \ell+\sqrt{2}$, we have that $|E^y_{st}|\rightarrow 0$ and $|Y_3|\rightarrow 1+\ell\sqrt{2}>0$ on $P_{\ell m}$.  Thus, it follows from parts (2) and (4) of Proposition \ref{extremal length proposition} that
\[ Ext_{P_{\ell m}}(Y_1,E^y_{st})\rightarrow\infty\ \ \mbox{as}\ \ m\rightarrow \ell+\sqrt{2}\ . \]
Therefore, by continuity there exists some intermediate
$\hat{m}=f(\ell)$ such that
\begin{equation}\label{Y1Eyst} Ext_{\Omega_{st}}(Y_1,E^y_{st})=Ext_{P_{\ell\hat{m}}}(Y_1,E^y_{st})\ . \end{equation}

Next, fix $1/\sqrt{2}<m\le\sqrt{2}$.  Arguing as before,
it follows that \begin{equation}\label{what} Ext_{P_{\ell m}}(Y_1,E_{st})\rightarrow 0\ \
\mbox{as}\ \ \ell\rightarrow 0\ . \end{equation} Furthermore, we have
\begin{equation}\label{hey} Ext_{P_{\ell m}}(Y_1,E_{st})\rightarrow\infty\ \ \mbox{as}\ \ \ell\rightarrow \infty\ .
\end{equation} To see this, consider the pentagon $(1/\ell)P_{\ell m}$.  As $\ell\rightarrow\infty$, this rescaled pentagon is such that $|Y_1|,|E_{st}|\rightarrow 0$ and $|E^x_{st}|\rightarrow 1$ as $\ell\rightarrow\infty$.  Statement (\ref{hey}) then follows by parts (2) and (4) of Proposition \ref{extremal length proposition}.  Therefore, there is an intermediate $\hat{\ell}=g(m)$
such that \begin{equation}\label{Y1Est}
Ext_{\Omega_{s,t}}(Y_1,E_{st})=Ext_{P_{\hat{\ell}m}}(Y_1,E_{st})\ .
\end{equation}

Statement (\ref{hey}) still holds for $m>\sqrt{2}$, but (\ref{what}) does not since $\ell=|E^x_{st}|$ no longer approaches zero.  More specifically, for fixed $m$ we have \[ Ext_{P_{\ell m}}(Y_1,E_{st})\rightarrow A_m>0\ \
\mbox{as}\ \ \ell\rightarrow m-\sqrt{2} \] and \[ A_m\rightarrow\infty\ \ \mbox{as}\ \ m\rightarrow\infty\ . \]  To see this last statement, consider the quadrilaterals $P_{\ell m}$ with $m>\sqrt{2}$ and $\ell=m-\sqrt{2}$.  As $m\rightarrow\infty$, rescaled quadrilaterals $(1/m)P_{\ell m}$ are such that $|Y_1|\rightarrow 0$ while $|Y_3|$, $|E_{st}|$ and $|E^x_{st}|$ all approach nonzero numbers.  Thus, it follows from parts (2) and (4) of Proposition \ref{extremal length proposition} that \[ A_m=Ext_{P_{\ell m}}(Y_1,E_{st})\rightarrow\infty\ \ \mbox{as}\ \ m\rightarrow\infty\ . \]

Furthermore, it follows from parts (5) and (6) of Proposition \ref{extremal length proposition} that $Ext_{P_{\ell m}}(Y_1,E_{st})$ increases for fixed $m$ and increasing $\ell$.  Thus, there is some $m_{st}<\infty$ such that function $g(m)=\hat{\ell}$ is only defined for $1/\sqrt{2}<m<m_{st}$, and the calculations and discussions above imply that \[ g(m)\rightarrow m_{st}-\sqrt{2}\ \ \mbox{as}\ \ m\rightarrow m_{st}\ . \]  Also, we must have that $g(m)$ is bounded away from infinity as $m\rightarrow 1/\sqrt{2}$, since otherwise we would have \[ Ext_{P_{\hat{\ell} m}}(Y_1,E_{st})\rightarrow\infty\ \ \mbox{as}\ \ m\rightarrow 1/\sqrt{2}\ . \]  This cannot happen because statement (\ref{Y1Est}) would not be true near $m=1/\sqrt{2}$.

Therefore, the graph of the continuous function $g(m)$ must intersect the graph of the continuous function $f(\ell)$ at some point $(\hat{\ell},\hat{m})$, and at this point we have that both (\ref{Y1Eyst}) and (\ref{Y1Est}) are true.

By the Riemann mapping theorem, there exists a conformal map
$\zeta_{st}$ from $\Omega_{s,t}$ onto $P_{\hat{\ell}\hat{m}}$, and we can normalize so
that
\begin{equation}\label{g1 zeta normalization} \zeta_{st}(Y_1)=Y_1\ \ \ \mbox{and}\ \ \ \zeta_{st}(Y_3)=Y_3\ .
\end{equation} Moreover, since (\ref{Y1Eyst}) holds, it follows
from part (6) of Proposition \ref{extremal length
proposition} that
\begin{equation}\label{Eystwoo} \zeta_{st}(E^y_{st})=E^y_{st}\ . \end{equation} Then, since (\ref{Y1Est}) also holds, it follows from part (6) of Proposition \ref{extremal length proposition} that \begin{equation}\label{Estwoo}
\zeta_{st}(E_{st})=E_{st}\ \ \mbox{and}\ \ \zeta_{st}(E^x_{st})=E^x_{st}. \end{equation}  Therefore,
from (\ref{g1 zeta normalization}), (\ref{Eystwoo}) and
(\ref{Estwoo}) we have that $\zeta_{st}$ is the desired conformal,
edge-preserving map of the proposition.\ \ \ $\Box$

We have now derived a parameterization $X^{st}$ on $\Omega_{st}$ for each $(s,t)\in C_4$.  To verify these parameterizations, we first check that $X^{st}(\Omega_{st})$ is still compact after the introduction of the edges $E^x_{st}$ and $E^y_{st}$.  For the one form $(d\zeta_{st})^2/dz$ to have a non-integrable singularity at one of the vertices $v_1=E^x_{st}\cap E_{st}$, $v_2=E^y_{st}\cap E_{st}$, $w_1=Y_1\cap E^x_{st}$ or $w_2=Y_3\cap E^y_{st}$, the domain $\Omega_{st}$ would need to have an angle of $\pi$ at $v_1$ or $v_2$ or an angle of $3\pi/2$ at $w_1$ or $w_2$.  The angle at $w_1$ and $w_2$ is $\pi/2$, and the angles at $v_1$ and $v_2$ are clearly less than $\pi$.  Thus, we have that \begin{equation}\label{C4compact} X^{st}(\Omega_{st})\ \ \mbox{is compact}\ . \end{equation}  Furthermore, we compute that the angle on the surface at $w_1$ and $w_2$ is $\pi$, so that $Y_1$, $Y_3$ meets the top, bottom edge, respectively, of $T_{st}$ tangentially.

To check $X^{st}$ on $\partial\Omega_{st}$, we first note that the calculations on $Y_1$, $Y_3$ and $E_{st}$ are the same as in the proof of part (3) of Theorem \ref{main theorem}, and so they are omitted here.  Parameterizing $E^x_{st}$ from $w_1$ to $v_1$ by $z_4(w)=w$, we have $\dot{z}_4\equiv 1$ and $d\zeta_{st}(\dot{z}_4)^2=i|d\zeta_{st}(\dot{z}_4)|^2$.  Computing, we have \begin{equation}\label{dX1stz4} dX^{st}_1(\dot{z}_4)=\frac{1}{2}\mbox{Re}((1-w^2)i|d\zeta_{st}(\dot{z}_4)|^2)=0\ , \end{equation} \begin{equation}\label{dX2stz4} dX^{st}_2(\dot{z}_4)=\frac{1}{2}\mbox{Re}(i(1+w^2)i|d\zeta_{st}(\dot{z}_4)|^2)=-\ \frac{1}{2}|d\zeta_{st}(\dot{z}_4)|^2(1+w^2)<0\ , \end{equation} and \begin{equation}\label{dX3stz4} dX^{st}_3(\dot{z}_4)=\mbox{Re}(wi|d\zeta_{st}(\dot{z}_4)|^2)=0\ . \end{equation}   Statements (\ref{dX1stz4}) - (\ref{dX3stz4}) imply $X^{st}$ maps $E^x_{st}$ monotonically onto a line segment in the direction of the $y$-axis.  Similarly, we have that $X^{st}$ maps $E^y_{st}$ monotonically onto a line segment in the direction of the $x$-axis.  Thus, the boundary of our fundamental piece $\hat{M}_{st}=X^{st}(\Omega_{st})$ is as expected.  As in the proof of cases (1) and (3), the surface $\hat{M}_{st}$ is a graph over its projection into the $xy$-plane.  Therefore, it has no self-intersections, and no intersections are introduced upon extension to a soap film spanning $T_{st}$.

If $(s,t)\in C_2^+\cup C_2^-$, then only one of the edges $E^y_{st}$ or $E^x_{st}$ is introduced.  So, the calculations are similar to and simpler than those above for the case $(s,t)\in C_4$.  Therefore, they are omitted.

\subsection{Proof of part (2)}

If $(s,t)\in F$, then from part (3) of Proposition \ref{Gammastprop} it follows that $\Gamma_{st}$, $\Gamma_1$ and $\Gamma_2$ have a mutual point of intersection.  This implies the Gauss image $\Omega_{st}$ is a point, which implies the fundamental piece $\hat{M}_{st}$ in this case should be planar.  Thus, we expect to find a flat cone over $T_{st}$ which is also a soap film.

Consider the cone over $T_{st}$ with vertex $P=(0,0,a)$, where \[ -(1/2)\sqrt{1-s^2-t^2}<a<(1/2)\sqrt{1-s^2-t^2}\ . \]  Denote by $E_1$, $E_2$ the Euclidean segment from $P$ to the vertex $(0,-t,\sqrt{1-s^2-t^2}/2)$, $(0,t,\sqrt{1-s^2-t^2}/2)$, respectively and $E_3$, $E_4$ the segment from $P$ to the vertex $(s,0,-\sqrt{1-s^2-t^2}/2)$, $(-s,0,-\sqrt{1-s^2-t^2}/2)$, respectively, and let $\theta_{jk}$ denote the angle between $E_j$ and $E_k$.  The direction of $E_1$ is \[ v_1=\left<0,-t,\frac{\sqrt{1-s^2-t^2}}{2}-a\right>\ , \] and the direction of $E_2$ is \[ v_2=\left<0,t,\frac{\sqrt{1-s^2-t^2}}{2}-a\right>\ .\]  Thus, the angle $\theta_{12}$ is given by \[ \cos\theta_{12}=\frac{v_1\cdot v_2}{|v_1||v_2|}=\frac{(\sqrt{1-s^2-t^2}/2-a)^2-t^2}{(\sqrt{1-s^2-t^2}/2-a)^2+t^2}\ , \] and $\theta_{12}=\arccos(-1/3)$ with $a<(1/2)\sqrt{1-s^2-t^2}$ if and only if \[ a=\frac{\sqrt{1-s^2-t^2}}{2}-\frac{t}{\sqrt{2}}\ . \]  Similarly, we have the angle $\theta_{34}$ equals $\arccos(-1/3)$ with $a>-(1/2)\sqrt{1-s^2-t^2}$ if and only if \[ a=\frac{s}{\sqrt{2}}-\frac{\sqrt{1-s^2-t^2}}{2}\ . \]  Thus, we have $\theta_{12}=\theta_{34}=\arccos(-1/3)$ if and only if \[ a=\frac{\sqrt{1-s^2-t^2}}{2}-\frac{t}{\sqrt{2}}=\frac{s}{\sqrt{2}}-\frac{\sqrt{1-s^2-t^2}}{2} \Leftrightarrow 3s^2+3t^2+2st=2 \] \[ \Leftrightarrow (s,t)\in F\ . \]  For this value of $a$ we also have \[ \cos\theta_{13}=\frac{v_1\cdot v_3}{|v_1||v_3|}=\frac{\left<0,-t,t/\sqrt{2}\right>\cdot\left<s,0,-s/\sqrt{2}\right>}{t\sqrt{3/2}s\sqrt{3/2}}=-1/3\ . \]  Therefore, since by symmetry we have $\theta_{13}=\theta_{23}=\theta_{14}=\theta_{24}$, it follows that, for $(s,t)\in F$, the flat cone over $T_{st}$ with vertex \[ P=(0,0,\sqrt{1-s^2-t^2}/2-t/\sqrt{2})=(0,0,s/\sqrt{2}-\sqrt{1-s^2-t^2}/2) \] is such that $P$ is a $T$-singularity and is hence a soap film.

\bibliographystyle{amsalpha}

\bibliography{liter}

\providecommand{\bysame}{\leavevmode\hbox to3em{\hrulefill}\thinspace}
\providecommand{\MR}{\relax\ifhmode\unskip\space\fi MR }
\providecommand{\MRhref}[2]{%
  \href{http://www.ams.org/mathscinet-getitem?mr=#1}{#2}
}
\providecommand{\href}[2]{#2}
\begin{thebibliography}{Alm76}

\bibitem[Ahl73]{Ah73}
L.~Ahlfors, \emph{Conformal invariants: topics in geometric function theory},
  McGraw-Hill Book Co., New York, 1973.

\bibitem[Alm76]{Alm76}
F.J. Almgren, Jr., \emph{Existence and regularity almost everywhere of
  solutions to elliptic variational problems with constraints}, Mem. AMS No.
  165 (1976).

\bibitem[Car54]{ca54}
C.~Carath\'{e}odory, \emph{Theory of functions of a complex variable}, vol.~2,
  Chelsea Publishing Company, New York, 1954.

\bibitem[Car76]{Dc76}
M.~Do Carmo, \emph{Differential geometry of curves and surfaces}, Prentice
  Hall, Paris, 1976.

\bibitem[HK97]{hk2}
D.~Hoffman and H.~Karcher, \emph{Complete embedded minimal surfaces of finite
  total curvature}, Encyclopedia of Mathematics, 1997, R. Osserman, editor,
  Springer Verlag, pp.~5--93.

\bibitem[HM06]{hmc1}
R.~Huff and J.~McCuan, \emph{Scherk-type capillary graphs}, J. Math. Fluid
  Mech. \textbf{8} (2006), no.~1, 99--119.

\bibitem[Huf05]{rh3}
R.~Huff, \emph{Soap films and {K}elvin's curved, truncated octahedron}, J.
  Geom. Anal. \textbf{15} (2005), no.~3, 425--443.

\bibitem[Huf06]{rh5}
\bysame, \emph{Soap films spanning rectangular prisms}, Geom. Dedicata
  \textbf{123} (2006), no.~1, 223--238.

\bibitem[LM94]{LaMo94}
G.~Lawlor and F.~Morgan, \emph{Paired calibrations applied to soap films,
  immiscible fluids, and surfaces or networks minimizing other norms}, Pac. J.
  Math. \textbf{166} (1994), no.~1, 55--83.

\bibitem[Rad33]{ra1}
T.~Rado, \emph{On the problem of plateau}, Springer Verlag, Berlin, 1933.

\bibitem[Tay76]{Tay76}
J.~Taylor, \emph{The structure of singularities in soap-bubble-like and
  soap-film-like minimal surfaces}, Ann. Math \textbf{103} (1976), 489--539.

\end{thebibliography}

\end{document}